\documentclass[a4paper,reqno]{amsart}

\usepackage[english]{babel}
\usepackage[utf8]{inputenc}
\usepackage{amssymb,amsmath,amsthm}
\usepackage[hidelinks]{hyperref}
\usepackage{amsfonts}
\usepackage{psfrag}
\usepackage{a4wide}
\usepackage[dvipsnames]{xcolor}
\usepackage{graphicx}
\usepackage{cancel}
\usepackage[normalem]{ulem}
\usepackage{accents}
\usepackage[foot]{amsaddr}

\DeclareGraphicsExtensions{.jpg,.jpeg,.pdf,.png,.eps}
\allowdisplaybreaks

\newcommand{\Ric}{\mathrm{Ric}}

\newcommand{\e}{\epsilon}

\renewcommand{\S}{\Sigma}

\newcommand{\n}{\nabla}

\renewcommand{\L}{\mathcal{L}}

\renewcommand{\o}{\omega}

\newcommand{\tr}{\mathrm{tr}}
\renewcommand{\a}{\alpha}

\renewcommand{\b}{\beta}
\renewcommand{\d}{\partial}

\renewcommand{\div}{\mathrm{div}}
\newcommand{\riem}{\mathrm{Riem}}

\renewcommand{\H}{\mathcal H}

\newcommand{\Riem}{\mathrm{Riem}}

\theoremstyle{plain}
\newtheorem{thm}{Theorem}[section]
\newtheorem{lemma}{Lemma}[section]
\newtheorem{cor}{Corollary}[section]
\newtheorem{definition}{Definition}[section]

\newtheorem{remark}{Remark}[section]

\newcommand{\R}[0]{\mathbb{R}}							
\newcommand{\N}[0]{\mathbb{N}}							

\DeclareFontFamily{OT1}{rsfs}{} \DeclareFontShape{OT1}{rsfs}{m}{n}{
	<-7> rsfs5 <7-10> rsfs7 <10-> rsfs10}{}
\DeclareMathAlphabet{\mycal}{OT1}{rsfs}{m}{n}

\title[Compact Cauchy horizons of constant non-zero surface gravity]
 {Symmetries of vacuum spacetimes with a compact Cauchy horizon of constant non-zero surface gravity}

\author{Oliver Petersen\,${}^{1}$}

\address[1]{Department of Mathematics, KTH Stockholm, Lindstedtsvägen 25, 11428 Stockholm, Sweden}
\email[1]{oliverlp@kth.se}

\author{Istv\'an R\'acz\,${}^{2}$}

\address[2]{
Wigner RCP \\
H-1121 Budapest \\
Konkoly Thege Mikl\'{o}s \'{u}t 29-33 \\
Hungary}
\email[2]{racz.istvan@wigner.hu}


\begin{document}
	\hbadness=100000
	\vbadness=100000

\begin{abstract}
We prove that any smooth vacuum spacetime containing a compact Cauchy horizon with surface gravity that can be normalised to a non-zero constant admits a Killing vector field. 
This proves a conjecture by Moncrief and Isenberg from 1983 under the assumption on the surface gravity and generalises previous results due to Moncrief-Isenberg and Friedrich-Rácz-Wald, where the generators of the Cauchy horizon were closed or densely filled a $2$-torus. 
Consequently, the maximal globally hyperbolic vacuum development of generic initial data cannot be extended across a compact Cauchy horizon with surface gravity that can be normalised to a non-zero constant.
Our result supports, thereby, the validity of the strong cosmic censorship conjecture in the considered special case.
The proof consists of two main steps.
First, we show that the Killing equation can be solved up to infinite order at the Cauchy horizon. 
Second, by applying a recent result of the first author on wave equations with initial data on a compact Cauchy horizon, we show that this Killing vector field extends to the globally hyperbolic region.
\end{abstract}

\maketitle

\tableofcontents
\begin{sloppypar}

\section{Introduction}
Penrose's strong cosmic censorship conjecture says that the maximal globally hyperbolic vacuum developments of generic initial data cannot be extended to a larger vacuum spacetime \cite{Penrose1979, Penrose1994, wald, Reall2018}.  
In spite of its importance, this intriguing conjecture is far from being proved. 
Indeed, the strong cosmic censorship conjecture always receives a foremost place in the list of the most important unresolved issues in Einstein's theory of gravity. 

If a maximal globally hyperbolic development was a proper open subset of a larger spacetime it could be extended across a Cauchy horizon. 
It is therefore of crucial importance to understand if the existence of a Cauchy horizon in a vacuum spacetime implies some restrictions on the geometry. Exactly this type of question was raised by Vince Moncrief---emanating from his comprehensive investigations of various cosmological spacetimes during the early $1980$'s  \cite{Moncrief1981a,Moncrief1982,Moncrief_FM_1980,Moncrief1981b}---in connection with spacetimes admitting a compact Cauchy horizon. 
He proposed that vacuum spacetimes with a compact Cauchy horizon necessarily admits a non-trivial Killing vector field in the globally hyperbolic region.
As the existence of a Killing vector field is a \emph{non-generic} property, such a statement would imply that spacetimes with compact Cauchy horizons necessarily are non-generic and thereby support the strong cosmic censorship conjecture in the considered special case.

The first remarkable step in applying this idea was made by Moncrief and Isenberg by proving the existence of a non-trivial Killing symmetry in \emph{analytic} electrovacuum spacetimes of dimension $4$, admitting a compact Cauchy horizon ruled by \emph{closed} generators \cite{MoncriefIsenberg1983}. 
One important step in the proof was to show that the surface gravity of any compact Cauchy horizon with closed generators could be normalised to zero (the degenerate case) or to a non-zero constant (the non-degenerate case). 
Moncrief and Isenberg conjectured in \cite{MoncriefIsenberg1983} that their results should hold without assuming analyticity and that the generators are closed.
Analyticity was later relaxed in the non-degenreate case by applying a combination of spacetime extensions and the characteristic initial value problem by Friedrich, R\'acz and Wald \cite{FRW1998}. 
Essentially the same techniques were also used in \cite{Racz2000} to generalise the proof to various coupled gravity matter systems.

Moncrief and Isenberg generalised their result in \cite{MoncriefIsenberg1983} to higher dimensions in \cite{MoncriefIsenberg2008} by proving the existence of a non-trivial Killing symmetry for higher dimensional analytic electrovacuum spacetimes admitting a compact Cauchy horizon with closed generators and for higher dimensional analytic electrovacuum stationary black hole spacetimes.
Their argument does not apply, in general, if the generators are not closed. 
A similar investigation for smooth higher dimensional electrovacuum black hole spacetimes, generalising the result in \cite{FRW1998} to higher dimensions, was done by Hollands, Ishibashi and Wald \cite{HIW2007}.

It is important to emphasise that in all the aforementioned results on compact Cauchy horizons, the generators were closed or densely filled a $2$-torus.
The purpose of the present paper is to provide generalisations of all the earlier results by removing both the analyticity assumption on the spacetime and the assumptions about the structure on the generators.
The only assumption in our work is that surface gravity can be normalised to a non-zero constant.
After the present paper appeared as a preprint, Bustamente and Reiris showed in \cite{BR2021} (see also the work by Gurriaran and Minguzzi in \cite{GM2021}) that if at least one generator of a compact Cauchy horizon in a vacuum spcaetime is incomplete, then the surface gravity can in fact be chosen to be a non-zero constant (which in turn implies that all generators are incomplete).
Our assumptions here therefore only exclude potential compact Cauchy horizons in vacuum spacetimes where all generators are complete.
However, to date no such example is known.

The main difficulty in dropping the assumption that the compact Cauchy horizon is ruled by closed generators or generators densely filling a $2$-torus is that each of the previously applied arguments rest on the use of Gaussian null coordinate systems. 
These coordinates are supposed to be defined in a neighbourhood of the Cauchy horizon (or---as they are applied in \cite{FRW1998,Racz2000}---in a neighbourhood of the universal cover of a subset of the Cauchy horizon). 
These Gaussian null coordinates have to be well-defined along the null generators, which cannot be guaranteed when some of the generators are non-closed. 
In order to avoid these difficulties we base our argument on a coordinate free framework introduced in \cite{Petersen2018}. 

Our proofs---besides relying heavily on the new result on wave equations with initial data on compact Cauchy horizons in \cite{Petersen2018}---are based on the following two fundamental new observations. 
First, a pair of coupled wave equations is used, relating a vector field to the Lie derivatives of the metric and the Ricci tensor with respect to that vector field. 
(A detailed derivation of these relations is given in a separate appendix, see in particular Lemma A.1). 
These simple equations allow to avoid the use of coordinate expressions for the Ricci curvature and its Lie derivatives, as used in \cite[Section~II.C]{MoncriefIsenberg1983}, and thus prove to be the main ingredient of our proof.
Second, in verifying that the Killing equation can be solved up to any order at the Cauchy horizon---using the aforementioned coupled wave equations---a first order linear and homogeneous ODE along the generators is derived for the norm of some specific components of the transverse (to the Cauchy horizon) derivatives of the Killing equation.
An important step in the proof is to use this ODE to show that the norm and, in turn, the pertinent components vanish.
The key point here---that allows to apply our argument to an arbitrary generator of the compact Cauchy horizon---is that a \emph{global} maximum principle can be applied to this ODE. 
In the case when the generators are closed, Moncrief and Isenberg used a corresponding maximum principle along each generator \cite{MoncriefIsenberg1983}. 
While the maximum principle does not apply to functions along non-closed generators, it applies to functions defined globally on the compact Cauchy horizon (see Lemma \ref{le: max_principle}).

\smallskip

Our results apply to smooth spacetimes $(M,g)$, i.e.\ connected time-oriented Lorentzian manifolds, of dimension $n+1 \geq 2$. The signature of the Lorentzian metric $g$ is fixed to be $(-, +, \dots, +)$. 
Consider now a closed acausal topological hypersurface $\S$ in $M$ (we do not require $\S$ to be compact), its Cauchy development $D(\S)$ is a globally hyperbolic submanifold in $(M, g)$. 
The boundary $\d D(\S)$ of $D(\S)$ is given by the disjoint union
\[
\d D(\S) = \H_+ \sqcup \H_-\,,
\]
where $\H_{\pm} := \overline{D^\pm(\S)} \backslash D^\pm(\S)$ denote the future and past Cauchy horizon, respectively.  
A lot is known about Cauchy horizons (see, e.g.~\cite[Chap. 6,8]{HawkingEllis1973} and \cite[Chap. 14]{O'Neill1983}). 
In particular, a Cauchy horizon is a lightlike Lipschitz hypersurface.
Denote by $\H$ the past or future Cauchy horizon of $\S$, and assume that $\H$ is non-empty and smooth.
Recall that \footnote{
Time-orientability of $M$ implies the existence of a nowhere vanishing time-like vector field $T$ on $M$.
Since $T$ is transversal to $\H$, we can define a one-form field $\b$ on a neighbourhood of $\H$ satisfying $\b(T) = 1$ and $\b(X) = 0$ for all $X \in T\H$. 
Then the vector field $V$ along $\H$, defined by $g(V, \cdot)|_\H = \b|_{\H}$, is a nowhere vanishing lightlike vector field tangent to  $\H$.}
there is a nowhere vanishing lightlike vector field $V$, tangent to $\H$, and a smooth function $\kappa$ such that
\[
	\n_VV = \kappa V\,.
\]

\begin{definition} \label{def: non-zero_constant_surface_gravity}
We say that the surface gravity can be normalised to a non-zero constant if there is a smooth nowhere vanishing lightlike vector field $V$ tangent to $\H$ such that 
\[
	\n_V V = \kappa V
\]
on $\H$ for some \bf{non-zero constant} $\kappa$.
\end{definition}

The integral curves of the lightlike vector field $V$---these are null geodesics---are called the \emph{generators} of $\H$.

\begin{remark}
If the surface gravity is constant and non-zero, then all generators are complete in one direction and incomplete in the other direction.
After the current paper appeared, the following remarkable converse statement was shown in \cite{BR2021} and \cite{GM2021}: If $\H$ is a smooth compact Cauchy horizon in a vacuum spacetime, containing one incomplete generator, then the surface gravity can be normalized to a non-zero constant.
\end{remark}

The first main result of the present paper is the following theorem.

\begin{thm}[Existence of an asymptotic Killing vector field] \label{thm: n = n}
Let $(M, g)$ be a spacetime containing a closed acausal topological hypersurface $\S$ and let $\H$ denote the past or future Cauchy horizon of $\S$.
Assume that $\H$ is compact, smooth and totally geodesic, and that the surface gravity can be normalised to a non-zero constant.
Assume that  $m \in \N$ and that 
\begin{equation}\label{eq: k-assumption}
	\n^k \Ric|_\H = 0
\end{equation}
for all $k \leq m$. 
Then there is a smooth non-trivial vector field $W$ on $\H \cup D(\S)$ such that
\begin{equation}\label{eq: n = n}
	\n^k \L_W g|_\H = 0\,,
\end{equation}
for all $k \leq m$.
Moreover, if \eqref{eq: k-assumption} holds for all $k \in \N_0$, then \eqref{eq: n = n} also holds for all $k \in \N_0$.
\end{thm}

Theorem \ref{thm: n = n} is proved in Section \ref{sec: asymptoticKilling}. It guarantees that there is a vector field satisfying the Killing equation up to any order. 
To prove the existence of a non-trivial Killing vector field in the globally hyperbolic region we have to propagate the asymptotic Killing field off the Cauchy horizon using wave equations.
That this can really be done is guaranteed by a recent result by the first author \cite[Thm.~1.6]{Petersen2018}. 
(In the analytic case the corresponding step is done by applying the Cauchy-Kovalewski theorem.)
We get the second main result of this paper:

\begin{thm}[Existence of a Killing vector field] \label{thm: existence Killing field}
Let $(M, g)$ be a Ricci-flat spacetime containing a closed acausal topological hypersurface $\S$ and let $\H$ denote the past or future Cauchy horizon of $\S$.
Assume that $\H$ is compact, smooth and totally geodesic, and that the surface gravity can be normalised to a non-zero constant.
Then there exists a smooth non-trivial Killing vector field $W$ on $\H \cup D(\S)$, i.e.
\[
	\L_W g = 0\,.
\]
$W$ is lightlike on $\H$ and spacelike in $D(\S)$ near $\H$, and any smooth extension of $W$ across $\H$ to the complement of $\overline{D(\S)}$ is timelike near $\H$.
\end{thm}

Theorem \ref{thm: existence Killing field} is proven in Section \ref{sec: existence Killing field}.
In fact, it actually is not necessary to assume that $\H$ is smooth and totally geodesic, as this is automatic by the following theorem by combining the work of Hawking \cite{Hawking1972} (see also \cite{HawkingEllis1973}), Larsson \cite{Larsson2014} and by Minguzzi \cite{Minguzzi2014,Minguzzi2015}:

\begin{thm}[Hawking \& Larsson \& Minguzzi]\label{Hawking-Larson-Minguzzi}
Assume that $(M, g)$ is a spacetime satisfying the null energy condition, i.e.~$\Ric(L, L) \geq 0$ for all lightlike vectors $L$ on $M$.
Let $\S \subset M$ be a closed acausal topological hypersurface and let $\H$ denote the past or future Cauchy horizon of $\S$.
If $\H$ is compact, then it is smooth and totally geodesic.
\end{thm}

Note that in applying \cite[Thm.~1.6]{Petersen2018} to solve wave equations with initial data given on a compact Cauchy horizon---as also indicated by  \cite[Counterexample~2.5]{Petersen2018}---the full asymptotic expansion of the candidate Killing vector field has to be used. 
In this respect our approach is different from the one by Alexakis, Ionescu and Klainerman in \cite{AIK2010}, where such an asymptotic expansion was not needed. 
Note, however, that the domains of existence of the solutions are also significantly different. 

There is yet another remarkable result by Isenberg and Moncrief which has immediate relevance to our results.
They proved in \cite{IsenbergMoncrief1992} that if there exists a non-trivial Killing vector field in a maximal globally hyperbolic vacuum development and the generators of the associated compact Cauchy horizon are \emph{non-closed}, then there must exist another non-trivial Killing symmetry. 
By combining \cite[Thm. 3]{IsenbergMoncrief1992} with our results, the following corollary---its proof is given at the end of Section \ref{sec: existence Killing field}---can be seen to hold.

\begin{cor}[Non-closed generators] \label{cor: multiple fields}
Let $(M, g)$ be a Ricci-flat spacetime containing a closed acausal topological hypersurface $\S$ and let $\H$ denote the past or future Cauchy horizon of $\S$.
Assume that $\H$ is compact, smooth and totally geodesic, and that the surface gravity can be normalised to a non-zero constant.
Assume further that at least one generator of $\H$ does not close and that $D(\S)$ is a maximal globally hyperbolic development.
Then there exist (at least) two distinct Killing vector fields on $D(\S)$, in fact the isometry group of $D(\S)$ must have an $S^1 \times S^1$ subgroup.
\end{cor}

For simplicity and definiteness, and also because of the novelty of the applied technical elements, in this paper only the vacuum problem is treated. 
Note, however, that the results by Moncrief-Isenberg and Friedrich-Rácz-Wald could be generalised in \cite{Racz2000} (see also \cite{Racz1999,Racz2001}) to various coupled gravity matter models. 
Note also that in such a circumstance not only the invariance of the metric but also the invariance of the matter field variables has to be demonstrated. 
Nevertheless, as the techniques applied by the second author in \cite{Racz2000} are analogous to those applicable in the pure vacuum case we strongly believe that our new results will also generalise to the inclusion of various matter models. Whether these expectations are valid remains to be investigated. 

Let us finally mention that after the present paper appeared as a preprint, the first author proved in \cite{P2019} that the Killing vector field in Theorem \ref{thm: existence Killing field} extends beyond the horizon as well.
The results in \cite{P2019} heavily rely on Theorem \ref{thm: existence Killing field}.

This paper is structured as follows.
Section 2 is to introduce the setup and prove Theorem \ref{thm: n = n}. 
The proof of the main result of this paper, Theorem \ref{thm: existence Killing field},---which is obtained by a combination of Theorem \ref{thm: n = n} and \cite[Thm. 1.6]{Petersen2018}---is included in Section \ref{sec: existence Killing field}. 
The derivation of the key identity is given in the Appendix. 
This identity is applied in proving Lemmas \ref{le: crucial equations} and \ref{baseVac} but they are proven in a general setting.

\section{Existence of an asymptotic Killing vector field} \label{sec: asymptoticKilling}

The ultimate goal of this section is to prove Theorem \ref{thm: n = n}. 
We assume in this section that $(M,g)$ is a spacetime and that $\H \subset M$ is a smooth, compact, totally geodesic Cauchy horizon with surface gravity that can be normalised to a non-zero constant.
For definiteness we assume that $\H$ is the past Cauchy horizon, the other case then follows by a time reversal.

For any subset $N \subset M$ and any vector bundle $F \to M$, we denote the space of smooth sections in $F$ defined on $N$ by
\[
	C^\infty(N, F).
\]
Our convention for the Riemannian curvature tensor is 
\[
	R(X, Y, Z, W)
		= g(\n_X \n_Y Z - \n_Y \n_X Z - \n_{[X, Y]}Z, W)
\]
and the Ricci curvature is given by
\[
	\Ric(X, Y)
		= \tr_g\left( R(X, \cdot, \cdot, Y) \right).
\]

\subsection{The null time function} \label{sec: time function}
Since $\kappa$ is a non-zero constant, the simple rescaling $V \mapsto \frac1\kappa V$ implies the rescaling $\kappa \mapsto 1$ in $\n_VV = \kappa V$. Thereby, without loss of generality, we shall assume that there exists a nowhere vanishing lightlike vector field $V$ tangent to $\H$ such that
\begin{equation}\label{kappa}
\n_VV = V
\end{equation}
holds everywhere on $\H$. 
As shown in \cite[Prop.~3.1]{Petersen2018}, a ``null time function'' can be constructed in a future neighbourhood of the past Cauchy horizon $\H$. 
The procedure is outlined below, whereas details can be found in \cite[Prop.~3.1]{Petersen2018}.  
As $\H$ is totally geodesic, it follows from \cite[Thm.~30]{Kupeli1987} that 
\[
	g(\n_X V, Y) = 0
\]
for all $X, Y \in T\H$. 
Therefore, there exists a smooth one-form $\o$ on $\H$ such that
\begin{equation} \label{eq: omega definition}
	\n_X V 
		= \o(X) V
\end{equation}
for all $X \in T\H$.
Note that, in virtue of \eqref{kappa}, $\n_V V = V=\o(V)\, V$ implying that $\omega(V) = 1$. 
Since $\o$ is nowhere vanishing it follows that $\ker(\o)$ is a vector bundle over $H$.
We get the splitting
\[
	T\H = \R V \oplus \ker(\o)\,.
\]
Using time-orientability of $M$, it can then be shown that there is a nowhere vanishing future pointing lighlike vector field $L$ on $\H$ such that $L \perp \ker(\o)$ and $g(L, V) = -1$.
It follows that $L$ is everywhere transverse to $\H$.
We may therefore define a local ``null frame'' $\{L, V, e_2, \hdots, e_n\}$ along $\H$ such that $\{e_2, \hdots, e_n\}$ is an orthonormal frame of $\ker(\o)$, and the metric takes the form
\[
	g|_\H = \begin{pmatrix} 0 & -1 & 0 \\ -1 & 0 & 0 \\ 0 & 0 & \delta_{ij}\end{pmatrix}\,.
\]
It is then shown in \cite[Prop. 3.1]{Petersen2018} that by flowing $\H$ along the lighlike geodesics emanating from $\H$ with tangent $L$ we get a foliation of an open subset of $\H$ in $\H \cup D(\S)$ by hypersurfaces diffeomorphic to $\H$.
More precisely, there exists an open set $U \subseteq \H \cup D(\S)$, with $\H \subset U$, and a unique smooth vector field $\d_t$, such that $\n_{\d_t}\d_t = 0$ and $\d_t|_\H = L$ and an associated smooth ``null time function'' $t:U \to [0,\e)$, such that $\d_t t = 1$ and such that $U$ is diffeomorphic to $[0, \e) \times \H$.
The time function $t:U \to [0,\e)$ is the ``time'' coordinate on $U \cong [0, \e) \times \H$. 
Let us emphasize, however, that we use the notation $\d_t$ even though it is not constructed as a part of a coordinate system.
The $t=const$ level hypersurfaces will be denoted by $\H_t$.
For $t>0$, $\H_t$ are Cauchy surfaces of $D(\S)$, whereas $\H_0 = \H$ is the Cauchy horizon.
In particular, we identify the horizon $\H$ with the set $\{t = 0\}$.

Define the vector field $W$ on $U$ by demanding that
\[
	[W, \d_t] = 0, \quad W|_{t = 0} = V.
\]
The vector field $W$ will indeed be the Killing vector field in Theorem \ref{thm: existence Killing field} (so far only defined near the horizon) and the Killing vector field to infinite order at the horizon as in Theorem \ref{thm: n = n}.
The remaining part of the paper is to prove this (and extend the Killing vector field to the entire globally hyperbolic region, as in Theorem \ref{thm: existence Killing field}).
Extend then the frame $\{e_2, \hdots, e_n\}$ of $\ker(\o)$ by Lie propagating them along $\d_t$, i.e.~by demanding that 
\[
	0 = [\d_t, e_2] = \hdots = [\d_t, e_n]\,.
\]
It follows that $(W, e_2, \hdots, e_n)$ is a local frame for $T\H_t$ for any $t \in [0, \e)$.
In order to express wave equations in terms of the null time function, the following lemma is essential.

\begin{lemma} \label{le: metric components}
Denote by $g_{\a \b} := g(e_\a, e_\b)$ the components of the metric, with respect to the frame  
\[
	\{e_0 := \d_t, e_1 := W, e_2, \hdots, e_n\}
\]
on $U$.
Let $g^{\a\b}$ denote the inverse of $g_{\a \b}$.
Then, for the components of the metric
\begin{align*}
	g_{00} &= g^{11} = 0\,, \\
	g_{01} &= g^{01} = -1\,, \\
	g_{0i} &= g^{1i} = 0\,, \quad i = 2, \hdots, n\,,
\end{align*}
hold. 
Moreover, we also have
\begin{align*}
	g_{11}|_{t = 0}
		&= g^{00}|_{t = 0}
		= 0, \\
	g_{1i}|_{t = 0} 
		&= g^{0i}|_{t = 0} = 0, \quad i = 2, \hdots, n.
\end{align*}
\end{lemma}
\begin{proof}
Since $\d_t$ is lightlike, $g_{00} = g(\d_t, \d_t) = 0$.
By construction, we also have 
\begin{align*}
	g_{01}|_{t = 0} &= -1\,, \\
	g_{0i}|_{t = 0} &= 0, \quad i = 2, \hdots, n.
\end{align*}
It also follows that for any $\a$, we have
\begin{align*}
	\d_tg_{0\a} 
		&= \d_tg(\d_t, e_\a ) = g(\n_{\d_t}\d_t, e_\a ) + g(\d_t, \n_{\d_t}e_\a ) \\*
		&= g(\d_t, \n_{e_\a} \d_t) = \tfrac12 \,\d_{e_\a}g_{00} = 0\,.
\end{align*}
Accordingly, we have that $g_{01} = -1$ and $g_{0i} = 0$ for $i=2, \hdots, n$, and, in turn, that
\begin{equation}
	g_{\a\b} = 
	\begin{pmatrix}
		0 & - 1 & 0 \\
		-1 & g_{11} & g_{1i} \\
		0 & g_{1i} & g_{ij}
	\end{pmatrix} \Rightarrow g^{\a\b} = 
	\begin{pmatrix}
		g^{00} & - 1 & g^{0i} \\
		-1 & 0 & 0 \\
		g^{0i} & 0 & g^{ij}
	\end{pmatrix}\,.
\end{equation}
This completes the proof of the first part of our assertions.

Since $V$ is lightlike, it also follows that ${g(W, e_i)|_{t = 0} = g(V, e_i|_{t= 0}) = 0}$ for $i = 1, \hdots, n$, which implies
\begin{equation}
	g_{\a\b}|_{t = 0} = 
	\begin{pmatrix}
		0 & - 1 & 0 \\
		-1 & 0 & 0 \\
		0 & 0 & g_{ij}|_{t = 0}
	\end{pmatrix} \Rightarrow g^{\a\b}|_{t = 0} = 
	\begin{pmatrix}
		0 & - 1 & 0 \\
		-1 & 0 & 0 \\
		0 & 0 & g^{ij}|_{t = 0}
	\end{pmatrix},
\end{equation}
for $i, j = 2, \hdots, n$, as claimed.
\end{proof}

\subsection{Properties of the null time function}

A certain curvature assumption at the horizon implies strong restrictions on the vector fields $\d_t$ and $V$ along $\H$.
We use the notation
\[
	\n_t := \n_{\d_t}
\]
where $t$ is the null time function as specified in Section \ref{sec: time function}.

\begin{lemma} \label{le: preparations}
Assume that $\Ric(Y, V)|_{t = 0} = 0$ for any $Y \in T\H$.
Then for any smooth vector field $X$ on $M$ such that ${X|_{t = 0} \in C^\infty(\H, \ker(\o))}$ and such that $[\d_t, X] = 0$, we have
\begin{align*}
	\n_X \d_t|_{t = 0} 
		&= \n_t X|_{t = 0} \in C^\infty(\H, \ker(\o)) \,, \\
	\n_V X|_{t = 0} 
		&= [V, X]|_{t = 0} \in C^\infty(\H, \ker(\o)) \,, \\
	\n_V \d_t|_{t = 0} 
		&= \n_t W|_{t = 0} = - \d_t|_{t = 0}
\end{align*}
and consequently
\begin{align*}
	\d_tg_{11}|_{t = 0} &= 2\,. \\
	\d_tg^{00}|_{t = 0} &= -2\,.
\end{align*}
\end{lemma}
\begin{proof}
Recall equation \eqref{eq: omega definition}, which is
\[
	\n_Y V = \o(Y) V
\]
for all $Y \in T\H$, which shows that $Y \in \ker(\o)$ if and only if $\n_Y V = 0$.
This will be used several times in this proof.

Since 
\begin{align*}
	g(\n_X\d_t, \d_t)|_{t = 0} 
		&= \tfrac12 X g(\d_t, \d_t)|_{t = 0} = 0\,, \\
	g(\n_X\d_t, V)|_{t = 0} 
		&= Xg(\d_t, V)|_{t = 0} - g(\d_t, \n_X V)|_{t = 0} = 0\,,
\end{align*}
we conclude by Lemma \ref{le: metric components} that $\n_X \d_t|_{t = 0} \in C^\infty\left(\H, \ker(\o)\right)$, proving the first assertion.

Since $\n_X V|_{t = 0} = 0$, proving the second assertion is equivalent to showing that $[X, V]|_{t = 0} \in C^\infty(\H, \ker(\o))$, or equivalently verifying that ${\n_{[X, V]}V|_{t = 0} = 0}$. 
In doing so notice, first, that since $\H$ is totally geodesic and as $[X, V]|_{t = 0} \in T\H$ it follows that $g(\n_{[X,V]}V, Z)|_{t = 0} = 0$ for all $Z \in T\H$. 
Hence it remains to show that $g(\n_{[X,V]}V, \d_t)|_{t = 0} = 0$ holds as well. 
For this, note first that
\begin{align} 
 \Ric(X, V)|_{t = 0} = -R(X, V, \d_t, V)|_{t = 0} + \sum_{i=2}^n R(X, e_i, e_i, V)|_{t = 0} \label{eq: ric(X,V)}\,,
\end{align}
for an arbitrary orthonormal frame $\{e_2,\dots,e_n\}$ in $\ker(\o)$.
The second term in equation \eqref{eq: ric(X,V)} vanishes, as
\begin{align*}
	R(X, e_i, e_i, V)|_{t = 0} &= - R(X, e_i, V, e_i)|_{t = 0} \\*
		&= - g(\n_X(\n_{e_i}V), e_i)|_{t = 0} + g(\n_{e_i}(\n_X V), e_i)|_{t = 0} \\
		& \qquad + g(\n_{[X, e_i]}V, e_i)|_{t = 0} \\
		&= 0\,,
\end{align*}
where, in the last step, the relations $\n_{e_i} V|_{t = 0} = 0 = \n_X V|_{t = 0}$ and $[X, V]|_{t = 0} \in T\H$ have been used. 
Evaluating the first term in equation \eqref{eq: ric(X,V)} at $t = 0$, we get
\begin{align*}
	- R(X, V, \d_t, V)|_{t = 0} 
		&= R(X, V, V, \d_t)|_{t = 0} \\*
		&= g(\n_X (\n_V V), \d_t)|_{t = 0} - g(\n_V (\n_X V), \d_t)|_{t = 0} \\
		& \quad - g(\n_{[X, V]} V, \d_t)|_{t = 0} \\*
		&= - g(\n_{[X, V]} V, \d_t)|_{t = 0}\,,
\end{align*}
where the relations $\n_V V|_{t = 0} = V|_{t = 0}$ and $\n_X V|_{t = 0} = 0$ have been used. 
Combining this with equation \eqref{eq: ric(X,V)} and using that, by assumption, $\Ric(X, V)|_{t = 0} = 0$ we get 
\[
	g(\n_{[X, V]} V, \d_t)|_{t = 0} = 0\,,
\]
and, in turn, that $[X, V]|_{t = 0} \in C^\infty(\H, \ker(\o))$, proving the second assertion.

To verify the third assertion, since $\d_t|_{t = 0} \perp \ker(\o)$, we get
\[
	g(\n_V\d_t, X)|_{t = 0} 
		= Vg(\d_t, X)|_{t = 0} - g(\d_t, \n_VX)|_{t = 0} 
		= 0\,,
\]
where the second term vanishes, in virtue of the assertion that has just been verified above. 
In addition, we further have
\begin{align*}
	g(\n_V\d_t, \d_t)|_{t = 0}
		&= \frac12 V g(\d_t, \d_t)|_{t = 0} 
		= 0 \,, \\
	g(\n_V\d_t, V)|_{t = 0} 
		&= Vg(\d_t, V)|_{t = 0} - g(\d_t, \n_VV)|_{t = 0} = 1\,.
\end{align*}
In virtue of Lemma \ref{le: metric components} and since $[\d_t, W] = 0$, it follows that $\n_V\d_t|_{t = 0} = - \d_t|_{t = 0} = \n_tW|_{t = 0}$.

Finally, we compute
\begin{align*}
	\d_tg_{11}|_{t = 0} 
		&= \L_{\d_t}g(V, V)|_{t = 0} = 2\,g(\n_V{\d_t}, V)|_{t = 0} = - 2\,g(\d_t, V)|_{t = 0} = 2\,, \\
	\d_tg^{00}|_{t = 0} 
		&= - g^{0\a}g^{0\b}\d_tg_{\a\b}|_{t = 0} 
		= - \d_tg_{11}|_{t = 0} 
		=  -2\,,
\end{align*}
verifying the last assertion.
\end{proof}

\subsection{The coupled wave equations} \label{keyrel}

Raising and lowering of indices will be signified---in non-self explaining situations---by the musical symbols ${}^\sharp$ and ${}_\flat$, respectively.
For any covariant $2$-tensor field $u$ on $M$, the symbol $\Box u$ is defined as
\[
	\Box u_{ab} := - \n^c\n_c u_{ab} \,.
\] 
Accordingly, in any (local) frame $\{e_0, \hdots, e_n\}$, defined on subsets of $M$, the term $\Box u$ is given as
\[
	\Box u = -g^{\a \b}(\n_{e_\a}\n_{e_\b}\,u - \n_{\n_{e_\a}e_\b}\,u) \,.
\]
For any covariant $2$-tensor $u$ the specific contraction $R_{a}{}^c{}_b{}^d \, u_{cd}$ of the Riemann tensor an $u$ will be denoted by $\Riem(u)$ 
\[
	\Riem(u)_{ab} := R_{a}{}^c{}_b{}^d \, u_{cd}\,.
\]
Finally $\div(V)$ will stand for $\n_aV^a$, whereas, for any covariant $2$-tensor $u$, $\div(u)$ will denote the contraction $\n^a u_{ab}$.

\medskip

The key relation, verified by the proof of Lemma~\ref{base} in the appendix, is the following:

\begin{lemma}\label{le: crucial equations}
Let $Z$ be a smooth vector field in $U$.
If $(\n_t)^k\Ric|_{t = 0} = 0$ for all $k \leq m+1$, then
\begin{equation}
	(\n_t)^k\Big(\Box \L_Zg - 2\Riem(\L_Zg) + \L_{\div\left(\L_Zg - \div(Z)g\right)^\sharp}g\Big)|_{t = 0} 
		= 0 \,, \label{eq: dabLg_k} 
\end{equation}
for all $k \leq m$.
\end{lemma}

Analogously, in the vacuum case, the following lemma can also be deduced from Lemma~\ref{base}.

\begin{lemma}\label{baseVac}
Assume that $(M,g)$ is a vacuum spacetime and let $Z$ be a smooth vector field in $U$. Then
	\begin{align}
	\Box\,\L_Zg - 2\,\Riem(\L_Zg) + \L_{\div\left(\L_Zg - \div(Z)g\right)^\sharp}g  
		& = 0 \,, \label{eq: LieVvecVac} \\
	\Box \,Z + \div\big(\L_Z g - \div(Z)\, g\big)^\sharp 
		& = 0 \,. \label{eq: VvecVac}
	\end{align}	
\end{lemma}

The next lemma plays a key role in the proof of Theorem \ref{thm: n = n}. In what follows it will be said that  a linear differential operator $P$ is \emph{differentiating along $\H_t$} if $Pu|_{\H_t}$ only depends on $u|_{\H_t}$ for all sections $u$, or, in other words, if $P$ does not involve $\n_t$-derivatives.

\begin{lemma} \label{le: k_derivative}
Let $Z$ be a smooth vector field in $U$.
For any $k \in \N_0$ and for any $X \in C^\infty(\H, T\H)$
\begin{align}
	(\n_t)^k &\left(\Box \L_Zg - 2\, \Riem(\L_Zg)\right)|_{t = 0} \nonumber \\
		&= 2\, \n_V (\n_t)^{k+1}\L_Zg|_{t = 0} + 2\,(k+1)\,(\n_t)^{k+1}\L_Zg|_{t = 0} \nonumber \\
		&\quad + S_k(\L_Zg|_{t = 0}, \hdots, (\n_t)^k\L_Zg|_{t = 0})\,, \label{eq: k_derivative_1st} \\
	(\n_t)^k &\div(\L_Zg - \div(Z)g)|_{t = 0}(X) \nonumber \\
		&= - (\n_t)^{k+1}\L_Zg|_{t = 0}(V, X) \nonumber \\
		& \quad + T_k(\L_Zg|_{t = 0}, \hdots, (\n_t)^k\L_Zg|_{t = 0})(X) \label{eq: k_derivative_2nd}
\end{align}
hold, where $S_k$ and $T_k$ are linear differential operators differentiating only along $\H$.
\end{lemma}

\begin{proof}
We start by verifying equation \eqref{eq: k_derivative_1st}. 
Note first that $(\n_t)^k(-2\Riem(\L_Zg))$ does not depend on $(\n_t)^{k+1}\L_Zg$.
Hence 
\[
	(\n_t)^k \left( - 2\, \Riem(\L_Zg)\right) = S^1_k(\L_Zg|_{\H_t}, \hdots, (\n_t)^k\L_Zg|_{\H_t})\,,
\] 
for some linear differential operator $S^1_k$ which is differentiating only along $\H_t$.
Let 
\[
	\{e_0 := \d_t, e_1 := W, e_2, \hdots, e_n\}
\]
be the local frame introduced in Section \ref{sec: time function}, where $\{e_2, \hdots, e_n\}$ is a frame for the vector bundle $\ker(\o)$.
In the next step, we will use the fact that
\[
	[\n_t, \n_{e_\a}] = R(\d_t, e_\a) + \n_{[\d_t, e_\a]} = R(\d_t, e_\a)
\]
where $R$ is the curvature tensor on $M$, realized here as a $(2,2)$-tensor. 
Consequently $[\n_t, \n_{e_\a}]$ is an endomorphism (a differential operator of order $0$).
Evaluating $(\n_t)^k \Box \L_Zg$ using this, we get
\begin{align*}
	&(\n_t)^k \Box \L_Zg \\
		&\quad = - (\n_t)^k\left(g^{\a \b}(\n_{e_\a}\n_{e_\b} - \n_{\n_{e_\a}e_\b})\L_Zg\right) \\
		&\quad = - [(\n_t)^k,g^{\a \b}] (\n_{e_\a}\n_{e_\b} - \n_{\n_{e_\a}e_\b})\L_Zg - g^{\a\b}[(\n_t)^k, \n_{e_\a}] \n_{e_\b} \L_Zg \\
		&\qquad - g^{\a\b} \n_{e_\a} [(\n_t)^k, \n_{e_\b}] \L_Zg + g^{\a\b} [(\n_t)^k, \n_{\n_{e_\a} e_\b}] \L_Zg + \Box (\n_t)^k \L_Zg \\
		&\quad = -k\d_t(g^{00})(\n_t)^{k-1}(\n_t)^2\L_Zg + S^2_k(\L_Zg|_{\H_t}, \hdots, (\n_t)^k\L_Zg|_{\H_t}) \\*
		&\quad \qquad + \Box (\n_t)^k \L_Zg \,,
\end{align*}
for some linear differential operator $S^2_k$, which is differentiating only along $\H_t$.
Before evaluating this expression at $t = 0$, note that since $\H$ is totally geodesic, we have that $\n_{e_\a}e_\b|_{t = 0} \in T\H$ for $\a, \b = 1, \hdots, n$.
Evaluating then $(\n_t)^k \Box \L_Zg$ at $t = 0$, in virtue of Lemma~\ref{le: metric components} and Lemma \ref{le: preparations}, we get 
\begin{align*}
	(\n_t)^k \Box \L_Zg|_{t = 0} 
		&= 2\,\n_V(\n_t)^{k+1}\L_Zg|_{t = 0} + 2\, (k+1)\,(\n_t)^{k+1}\L_Zg|_{t = 0} \\
		& \quad + S^3_k(\L_Zg|_{t = 0}, \hdots, (\n_t)^k\L_Zg|_{t = 0})\,,
\end{align*}
for some linear differential operator $S^3_k$ which is only differentiating along $\H$, verifying equation \eqref{eq: k_derivative_1st}. 

In verifying equation \eqref{eq: k_derivative_2nd}, assume that $X$ is Lie propagated along $\d_t$, i.e.~$[\d_t, X] = 0$.
It follows that $X \in T\H_t$ for all $t \in [0, \e)$.
Using $\div(\div(Z)g)(X) = \frac12 X \tr_g(\L_Zg)$ we immediately get that
\[
	(\n_t)^k\div\big(\div(Z)g\big)(X)|_{t = 0} = T_k^1(\L_Zg|_{t = 0}, \hdots, (\n_t)^k\L_Zg|_{t = 0})(X)\,,
\]
for some linear differential operator $T^1_k$ which is only differentiating along $\H$.
Analogously, we also have that
\begin{align*}
	(\n_t)^k\div(\L_Zg)(X) &= (\n_t)^k(g^{\a \b}\n_{e_\a}\L_Zg(e_\b, X)) \\
		&= g^{0 \a}((\n_t)^{k+1}\L_Zg)(e_\a, X) \\
		& \quad + T^2_k(\L_Zg|_{\H_t}, \hdots, (\n_t)^k\L_Zg|_{\H_t})(X)\,,
\end{align*}
for some linear differential operator $T^2_k$ which is only differentiating along $\H_t$.
Evaluating at $t = 0$, in virtue of Lemma \ref{le: metric components}, we get
\begin{align*}
	(\n_t)^k\div(\L_Zg)(X)|_{t = 0} 
		&= - ((\n_t)^{k+1}\L_Zg)(V, X)|_{t = 0} \\
		&\quad + T^2_k|_\H (\L_Zg|_{t = 0}, \hdots, (\n_t)^k\L_Zg|_{t = 0})(X)\,.
\end{align*}
Combining these observations completes the verification of equation \eqref{eq: k_derivative_2nd}.
\end{proof}

\subsection{The key lemma}\label{twolemmas}

The proof of the vanishing of various components of the Killing equation on $\H$ relies heavily on the following observation:

\begin{lemma}\label{le: max_principle}
Assume that $\mathfrak{a}$ is a smooth symmetric $2$-tensor field on $\H$ and $\b$ is a nowhere vanishing function, such that
\begin{equation} \label{eq: a_diff_eq}
	\n_V \mathfrak{a}(X,Y) + \b\, \mathfrak{a}(X,Y) = 0 
\end{equation}
for all $X, Y \in \ker(\o)$.
Then $\mathfrak{a}(X, Y) = 0$ for all $X, Y \in \ker(\o)$.
\end{lemma}
\begin{proof}
The proof relies on the fact that $\ker(\o) \subset T\H$ is a Riemannian subbundle. 
Thus, in particular,  $g$ is positive definite on $\ker(\o)$, we let $\mathfrak g$ denote its restriction to $\ker(\o)$.
This induces a positive definite metric  $\mathfrak g$ on the space 
\[
	\ker(\o)^*\otimes_{{sym}} \ker(\o)^* \subset \ker(\o)^*\otimes \ker(\o)^*,
\]
of symmetric $2$-tensor fields  by making use of the inverse of $g$. Using the abstract index notation $\mathfrak g(\mathfrak a, \mathfrak a)$ can be given as 
\[
\mathfrak g(\mathfrak a, \mathfrak a) = g^{ik}g^{jl}\,{\mathfrak a}{}_{ij}{\mathfrak a}{}_{kl}\,,
\]
where the indices run over $2, \hdots, n$.
In any local $g$-orthonormal frame $\{e_2, \hdots, e_n\}$ in $\ker(\o)$, $\mathfrak{g}(\mathfrak{a}, \mathfrak{a})$ can then be expressed as
\[
\mathfrak{g}(\mathfrak{a}, \mathfrak{a})=\sum_{i,j = 2}^n \mathfrak{a}(e_i, e_j)^2\,.
\]
Differentiating this along $V$, we get
\begin{align*}\label{eq: ODE}
	V \mathfrak{g} (\mathfrak{a}, \mathfrak{a}) 
		&= \sum_{i,j = 2}^n V (\mathfrak{a}(e_i, e_j)^2) \\
		&= 2 \,\sum_{i,j = 2}^n (V \mathfrak{a}(e_i, e_j))\,\mathfrak{a}(e_i, e_j) \\
		&= 2\, \sum_{i,j = 2}^n \Big( \n_V \mathfrak{a}(e_i, e_j)\,\mathfrak{a}(e_i, e_j) + \mathfrak{a}(\n_V e_i, e_j)\,\mathfrak{a}(e_i, e_j) \\
		& \qquad \qquad \qquad + \mathfrak{a}(e_i, \n_Ve_j)\,\mathfrak{a}(e_i, e_j) \Big)\,.
\end{align*}
As verified by Lemma \ref{le: preparations} we have $\n_V e_i \in \ker(\o)$ for $i=2,\dots,n$, which implies 
\begin{align*}
	\sum_{i, j = 2}^n \mathfrak{a}(\n_V e_i, e_j)\, \mathfrak{a}(e_i, e_j)
		& = \sum_{i, j,k = 2}^n g(\n_Ve_i, e_k)\, \mathfrak{a}(e_k, e_j)\,\mathfrak{a}(e_i, e_j) \\
		& = - \sum_{i, j,k = 2}^n g(e_i, \n_V e_k)\, \mathfrak{a}(e_k, e_j)\,\mathfrak{a}(e_i, e_j) \\
		& =  - \sum_{i, k = 2}^n \mathfrak{a}(e_k, e_j)\,\mathfrak{a}(\n_V e_k, e_j)\,.
\end{align*}
By combining all the above observations, in virtue of \eqref{eq: a_diff_eq} we get 
\[
	V \mathfrak{g}(\mathfrak{a}, \mathfrak{a}) = 2\,\mathfrak{g}(\n_V\mathfrak{a}, \mathfrak{a}) = - 2\b \,\mathfrak{g}(\mathfrak{a}, \mathfrak{a})\,.
\]
Since $\H$ is compact, the scalar function $\mathfrak{g}(\mathfrak{a}, \mathfrak{a})$ must attain its maximum and minimum.
We necessarily have that $V \mathfrak{g}(\mathfrak{a}, \mathfrak{a}) = 0$ at these locations and since $\b \neq 0$ also that $\mathfrak{g}(\mathfrak{a}, \mathfrak{a}) = 0$.
This implies then that $\mathfrak{g}(\mathfrak{a}, \mathfrak{a}) = 0$ everywhere on $\H$. Finally, as $\mathfrak{g}$ is a positive definite metric on $\ker(\o)^*\otimes_{sym} \ker(\o)^*$ our assertion $\mathfrak{a} = 0$ follows as claimed.
\end{proof}

\subsection{Finishing the proof}\label{finishing}

The proof of Theorem \ref{thm: n = n} will be given by an induction argument.

\medskip

We start by showing that the components $(\n_t)^k\L_Wg(\d_t, \cdot)|_{t = 0}$ can be expressed in terms of lower order derivatives of $\L_Wg$.

\begin{lemma} \label{le: the d_t component}
For any $k \in \N$, we have
\[
	(\n_t)^{k+1} \L_Wg(\d_t, \cdot)|_{t = 0} 
		= K_k(\L_Wg|_{t = 0}, \hdots, (\n_t)^k\L_Wg|_{t = 0})\,,
\]
where $K_k$ is a linear differential operator along $\H$.
\end{lemma}
\begin{proof}
First we show that $(\L_{\d_t})^{k+1}g(\d_t, \cdot) = 0$ for all $k \in \N$. 
To see this we shall use a local frame of the type $\{e_0 := \d_t, e_1 := W, e_2, \hdots, e_n\}$ as in Lemma \ref{le: metric components}. 
Since $[\d_t, e_\alpha] = 0$ for $\alpha = 0, \hdots, n$, it follows that 
\[
	\L_{\d_t}g(\d_t, e_\alpha) = \d_t g(\d_t, e_\alpha) = g(\d_t, \n_t e_\alpha) = \tfrac12\, e_\a g(\d_t, \d_t) = 0
\]
for $\alpha =0, \hdots, n$.
Hence, we also have that 
\[
	(\L_{\d_t})^{k+1}g(\d_t, e_\alpha) = (\d_t)^k\L_{\d_t}g(\d_t, e_\alpha) = 0
\]
for $\a = 0, \hdots, n$ as claimed.
By applying $[\d_t, W] = 0$, we also get 
\begin{align*}
	(\L_{\d_t})^{k+1}\L_Wg (\d_t, e_\a) 
		&= \L_W(\L_{\d_t})^{k+1}g(\d_t, e_\a) \\
		&= W(\L_{\d_t})^{k+1}g(\d_t, e_\a) - (\L_{\d_t})^{k+1}g(\d_t, [W, e_\a]) \\
		&= 0\,,
\end{align*}
for $\a = 0, \hdots, n$, showing that $(\L_{\d_t})^{k+1}\L_Wg(\d_t, \cdot) = 0$ as claimed.
The proof is completed by observing that $(\L_{\d_t})^{k+1} - (\n_t)^{k+1}$ is a linear differential operator of order $k$.
\end{proof}

The first step in our inductive proof the following:

\begin{lemma}{\rm [The case $m = 0$]} \label{le: n = 0}
Assume that $\Ric|_{t = 0} = 0$. 
Then
\begin{align*}
	\L_W g|_{t = 0} 
		&= 0.
\end{align*}
\end{lemma}
\begin{proof}
Since $\H$ is totally geodesic, it follows that 
\[
	\L_Wg|_{t = 0}(X, Y) = g(\n_XV, Y)|_{t = 0} + g(X, \n_YV)|_{t = 0} = 0,
\]
for all $X, Y \in T\H$.
It therefore remains to show that $\L_Wg|_{t = 0}(\d_t, \cdot) = 0$ as well. 
In doing so note first that by Lemma \ref{le: preparations} we have
\begin{align*}
	\L_W g|_{t = 0}(\d_t,\d_t) 
		&= 2 g(\n_t W, \d_t)|_{t = 0} 
		= - 2 g(\d_t, \d_t)|_{t = 0} 
		= 0 \,, \\
	\L_W g|_{t = 0}(\d_t, V) 
		&= g(\n_t W, V)|_{t = 0} + g(\n_V V, \d_t)|_{t = 0} \\*
		&= g(-\d_t, V)|_{t = 0} + g(V, \d_t)|_{t = 0} = 0 \,,
\end{align*}
and, for any smooth vector field $X$ such that $X|_{t = 0} \in C^\infty(\H, \ker(\o))$, that
\begin{align*}
	\L_Wg|_{t = 0}(\d_t, X) 
		&= g(\n_t W, X)|_{t = 0} + g(\d_t, \n_X V)|_{t = 0} 
		= -g(\d_t, X)|_{t = 0} = 0 \,.
\end{align*}
This completes the verification of $\L_Wg|_{t = 0} = 0$.
\end{proof}

\begin{lemma} \label{le: horizontal}
Assume that $\Ric|_{t = 0} = 0$. 
Then
\begin{align*}
	\left( \div(\L_Wg - \div(W) g \right) (X)|_{t = 0}
		&= 0 \,,
\end{align*}
for all $X \in T\H$.
\end{lemma}
\begin{proof}
By applying equation \eqref{eq: k_derivative_2nd} with $k = 0$ and $Z = W$, we get 
\begin{align*}
	 \div\left(\L_W g - \div(W) g\right)(X)|_{t = 0}
		&= - \n_t \L_Wg (V, X) |_{t = 0}
\end{align*}
for all $X \in T\H$.
By Lemma \ref{le: preparations}, it also follows that
\begin{align*}
	- \n_t \L_Wg(V, V)|_{t = 0}
		&= - 2g(\n^2_{\d_t, V}W, V)|_{t = 0} \\
		&= - 2 R(\d_t, V, V, V)|_{t = 0} - 2g(\n_V \n_t W, V)|_{t = 0} \\
		&\qquad + 2 g(\n_{\n_V\d_t} W, V)|_{t = 0} \\
		&= 2g(\n_V\d_t, V)|_{t = 0} - 2g(\n_tW, V)|_{t = 0} \\
		&= 0\,,
\end{align*}
since $[\d_t, W] = 0$.
Using that $\Ric(V, X)|_{t = 0} = 0$ and Lemma \ref{le: preparations}, we get then for any $X \in \ker(\o)$:
\begin{align*}
	- \n_t\L_W g(V, X)|_{t = 0}
		&= g(\n^2_{\d_t, V}W, X)|_{t = 0} + g(\n^2_{\d_t, X}W, V)|_{t = 0} \\
		&= R(\d_t, V, V, X)|_{t = 0} + R(\d_t, X, V, V)|_{t = 0} \\*
		& \qquad + g(\n^2_{V, \d_t}W, X)|_{t = 0} + g(\n^2_{X, \d_t}W, V)|_{t = 0} \\
		&= \sum_{i = 2}^n R(e_i, V, e_i, X)|_{t = 0} + g(\n_V \n_t W, X)|_{t = 0} \\
		& \qquad - g(\n_{\n_V\d_t}W, X)|_{t = 0} + g(\n_X \n_t W, V)|_{t = 0} \\
		& \qquad - g(\n_{\n_X \d_t} W, V)|_{t = 0} \\
		&= - \sum_{i = 2}^n R(e_i, X, V, e_i)|_{t = 0}  - g(\n_V \d_t, X)|_{t = 0} \\
		& \qquad + g(\n_t V, X)|_{t = 0} - g(\n_X\d_t, V)|_{t = 0} \\
		& \qquad - (\n_X \d_t)g(V, V)|_{t = 0} \\
		&= - \sum_{i = 2}^n g(\n_{e_i} \n_X V - \n_X \n_{e_i}V - \n_{[e_i, X]}V, e_i)|_{t = 0} \\
		&= 0 \,.
\end{align*}
Combining these observations completes the proof.
\end{proof}

We may now proceed with the induction step:

\begin{lemma}{\rm [The induction step]} \label{le: induction_step}
Let $m \in \N$. Assume that $(\n_t)^{k}\Ric|_{t = 0} = 0$ for all $k \leq m$, and that
\begin{align}
	(\n_t)^k \L_W g |_{t = 0} &= 0\,, \label{eq: ind_assumption_1}
\end{align}
for all $k \leq m-1$.
Then 
\begin{align*}
	(\n_t)^m \L_W g |_{t = 0} &= 0\,.
\end{align*}
\end{lemma}

\begin{proof}
Note first that Lemma \ref{le: the d_t component} implies that
\[
	(\n_t)^m \L_Wg|_{t = 0}  (\d_t, \cdot) = 0\,.
\]
It therefore suffices to show the vanishing of the remaining components.

The key equation in the proof is obtained by combining assumptions \eqref{eq: ind_assumption_1} with equations \eqref{eq: dabLg_k} and \eqref{eq: k_derivative_1st}.
We get
\begin{align}
	0 	
		&= 2 \n_V \n_t^k\L_Wg|_{t = 0} + 2k \n_t^k \L_Wg|_{t = 0} \nonumber \\
		&\qquad - (\n_t)^{k-1} \L_{\div \left( \L_Wg - \div(W)g\right)^\sharp}g|_{t = 0}\,, \label{eq: dabLg_n}
\end{align}
for any $1 \leq k \leq m$.

Let us treat the case $m = 1$ separately.
First, equation \eqref{eq: k_derivative_2nd}, with $k = 0$, combined with Lemma \ref{le: horizontal} proves that
\[
	\n_t\L_Wg(V, X)|_{t = 0} = 0 \,,
\]
for any $X \in T\H$.
Lemma \ref{le: horizontal} implies that for any $X, Y \in T\H$, 
\begin{align*}
	\L_{\div \left( \L_Wg - \div(W)g\right)^\sharp}g(X, Y)|_{t = 0}
		&= \n_X \div \left( \L_Wg - \div(W)g\right)(Y)|_{t = 0} \\
		&\qquad + \n_Y \div \left( \L_Wg - \div(W)g\right)(X)|_{t = 0} \\
		&= X \div \left( \L_Wg - \div(W)g\right)(Y)|_{t = 0} \\
		&\qquad - \div \left( \L_Wg - \div(W)g\right)(\n_X Y)|_{t = 0} \\
		&\qquad + Y \div \left( \L_Wg - \div(W)g\right)(X)|_{t = 0} \\
		&\qquad - \div \left( \L_Wg - \div(W)g\right)(\n_Y X)|_{t = 0} \\
		&= 0,
\end{align*}
where we have used that $\n_X Y, \n_Y X \in T\H$, since $\H$ is totally geodesic.
Inserting this into \eqref{eq: dabLg_n}, with $k = 1$, we note that
\[
	\n_V \n_t\L_Wg(X, Y)|_{t = 0} + \n_t \L_Wg(X, Y)|_{t = 0} = 0
\]
for all $X, Y \in T\H$.
By Lemma \ref{le: max_principle}, we conclude that
\[
	\n_t \L_Wg(X, Y)|_{t = 0} = 0
\]
for all $X, Y \in \ker(\o)$.
Hence $\n_t \L_Wg|_{t = 0} = 0$, which is the claim for $m = 1$.

We may now assume that $m \geq 2$.
We first claim that
\begin{equation} \label{eq: claim}
	\n_t^k \div \left( \L_Wg - \div(W)g \right)|_{t = 0} = 0
\end{equation}
for all $k \leq m-1$.
For $k \leq m-1$, this is immediate from \eqref{eq: k_derivative_2nd}.
The idea for $k = m-1$ is to insert $\d_t$ into equation \eqref{eq: dabLg_n}.
Applying the induction assumption and equation \eqref{eq: dabLg_n}, with $k = m-1$, we note that
\[
	(\n_t)^{m-2} \L_{\div \left( \L_Wg - \div(W)g\right)^\sharp}g|_{t = 0} = 0\,.
\]
Inserting $\d_t$, we conclude that
\[
	\n_t^{m-1}\div \left( \L_Wg - \div(W)g\right)|_{t = 0} = 0
\]
as claimed.
Equation \eqref{eq: k_derivative_2nd}, with $k = m-1$, now gives
\[
	\n_t^m\L_Wg|_{t = 0}(V, X) = 0 \,,
\]
for any $X \in T\H$.
Moreover, \eqref{eq: claim} implies together with equation \eqref{eq: dabLg_n}, with $k = m$, that
\[
	\n_V \n_t^m\L_Wg|_{t = 0} + m \n_t^m \L_Wg|_{t = 0} = 0
\]
for all $X, Y \in T\H$.
Lemma \ref{le: max_principle} thus implies that
\[
	\n_t^m\L_Wg|_{t = 0}(X, Y) = 0 \,,
\]
for all $X, Y \in \ker(\o)$.
This completes the proof.
\end{proof}

\begin{proof}[Proof of Theorem \ref{thm: n = n}]
The proof follows by induction using Lemma \ref{le: n = 0} and Lemma \ref{le: induction_step}.
\end{proof}

\section{Existence of a Killing vector field} \label{sec: existence Killing field}

The purpose of this section is to prove our main result, Theorem \ref{thm: existence Killing field}, by combining Theorem \ref{thm: n = n} and \cite[Thm.~1.6]{Petersen2018}.

\begin{proof}[Proof of Theorem \ref{thm: existence Killing field}]
Note first that by Theorem \ref{thm: n = n}, the vector field $W$, defined on the one-sided neighbourhood $U=[0, \e) \times \H$ of $\H$, satisfies
\[
	\n^{m} \L_W g |_\H 
		= 0
\]
for all $m \in \N_0$, and thus, by \eqref{eq: VvecVac} and \eqref{eq: k_derivative_2nd}, also
\begin{equation} \label{eq: asymptotic Box W}
	\n^{m} \Box W|_\H = 0
\end{equation}
for all $m \in \N_0$.
The idea is now to use $W$ as initial data for a characteristic initial value problem, for which well-posedness was proven by the first author in \cite[Thm.~1.6]{Petersen2018}.
The solution to the characteristic initial value problem, which we call $\hat W$, will be shown to be a Killing vector field, which coincides with the locally constructed $W$ on $U$.
The vector field $\hat W$ will thus extend $W$ to the entire globally hyperbolic region, proving Theorem \ref{thm: existence Killing field}.

By Theorem 1.6 of \cite{Petersen2018} with $P = \Box$, $f = 0$ and $w^N = W$, using \eqref{eq: asymptotic Box W},  there exists a unique vector field $\hat W \in C^\infty(\H \cup D(\S))$ such that 
\[
	\Box \hat W 
		= 0
\]
on $\H \cup D(\S)$, and such that
\begin{align}
	\n^m \hat W|_\H 
		&= \n^m W|_\H \label{eq: W V coincide on H}
\end{align}
for all $m \in \N_0$.
Inserting $\Box \hat W = 0$ into \eqref{eq: VvecVac} and \eqref{eq: LieVvecVac}, we get 
\begin{align}
	\Box\L_{\hat W}g - 2\Riem(\L_{\hat W}g) 
		&= 0 \,. \label{eq: wave for LWg W}
\end{align}
Note also that equation \eqref{eq: W V coincide on H}, along with Theorem \ref{thm: n = n}, implies that 
\begin{align*}
	\n^m \L_{\hat W} g|_\H 
		= \n^m \L_Wg|_\H
		= 0\,,
\end{align*}
for all $m \in \N_0$, which in virtue of \cite[Cor. 1.8]{Petersen2018} and \eqref{eq: wave for LWg W} implies that $\L_{\hat W} g = 0$ on $\H \cup D(\S)$. 

We therefore know that $\hat W$ is a Killing vector field on $\H \cup D(\S)$ and would like to show that
\[
	\hat W|_{U} = W,
\]
i.e.\ that $\hat W$ is really extending $W$ to the globally hyperbolic region.
Since $\hat W|_\H = V = W|_\H$ and $[\d_t, W] = 0$, it suffices to show that 
\begin{equation} \label{eq: dt hat W}
	[\d_t, \hat W] 
		= 0.
\end{equation}
By \eqref{eq: W V coincide on H}, we in particular know that $\n_t \hat W|_\H = \n_t W|_\H$ and hence
\[
	\L_{\hat W} \d_t|_{\H}
		= [\d_t, \hat W]|_\H = 0,
\]
so the lightlike vector $\d_t|_\H$ is invariant under the flow of $\hat W$.
Recall now that $\n_{\d_t} \d_t = 0$, i.e.\ the integral curves of $\d_t$ are the geodesics in $U$ emanating from the $\hat W$-invariant vector field $\d_t|_\H$.
Since $\hat W$ is a Killing vector field, these geodesics, and hence $\d_t$, are $\hat W$-invariant, which proves \eqref{eq: dt hat W}.
We conclude that indeed
\[
	\hat W|_U
		= W,
\]
as claimed.
In particular, this shows that $W$ is indeed a Killing vector field.

Finally, by Lemma \ref{le: preparations}, we know that
\[
	\d_t g(W, W)|_\H
		= \d_t g_{11}|_{t = 0}
		= 2,
\]
verifying that $W$ is spacelike in a future neighbourhood of $\H$ and that any smooth extension of $W$ to the complement of $\overline{D(\S)}$ across $\H$, is timelike. 
\end{proof}

We finish by proving Corollary \ref{cor: multiple fields}.

\begin{proof}[Proof of Corollary \ref{cor: multiple fields}]
In the proof of Theorem \ref{thm: n = n}, we see that the Killing vector field $W$ satisfies $[\d_t, W] = 0$ and $W|_\H = V$, which implies that $W$ is tangent to the hypersurfaces $\H_t := \{t\} \times \H$ in $U$. 
In virtue of \cite[Prop.~3.1]{Petersen2018}, the hypersurfaces $\H_t$ are Cauchy surfaces in the maximal globally hyperbolic spacetime $D(\S)$ for any $t \in (0, \e)$.
Accordingly, $W$ is a spacelike Killing vector field on $(-\e, \e) \times \H$, which leaves the individual Cauchy surfaces $\H_t$ invariant for any $t \in (0, \e)$. 
Therefore \cite[Thm. 3]{IsenbergMoncrief1992} can be applied to complete the proof of Corollary \ref{cor: multiple fields}.
\end{proof}

\section*{Appendix}
\renewcommand{\theequation}{A.\arabic{equation}}
\setcounter{equation}{0}
\renewcommand{\thelemma}{A.\arabic{lemma}}
\setcounter{lemma}{0}
\setcounter{remark}{0}
The purpose of this appendix is to verify the key relations applied in our argument. 
Notably, \eqref{eq: dabLg_k}, \eqref{eq: LieVvecVac} and \eqref{eq: VvecVac} are just special cases of certain identities which hold for any sufficiently regular vector field on any differentiable manifold $M$ endowed with a semi-Riemannian metric $g$. 
The signature of $g$ does not play any role so it could be arbitrary. 
Let us also emphasize that the derivations in this appendix make no use of Einstein's equations or any other filed equation.

We use the notation introduced in Subsection\,\ref{keyrel} together with $\Ric^\sharp(Z)$ denoting $\Ric{_a}^bZ_b$ and $\Ric^\sharp(u)$ denoting $\Ric{_a}{^c}u_{cb}$ for any vector field $Z$ and covariant $2$-tensor $u$.

\begin{lemma}\label{base}
For any smooth vector field $Z$ on an $n$-dimensional differentiable manifold $M$ endowed with a semi-Riemannian metric $g$ the following identities hold
\begin{align}
	\Box \,Z + \div(\L_Z g - \div(Z)\, g)^\sharp & = \Ric^\sharp(Z) \label{eq: Vvec} \\ 
	\Box\L_Zg - 2\,\Riem(\L_Zg) - \L_{\Box Z}g  
		& = 2\,\L_Z \Ric  - \L_{\Ric^\sharp\,(Z)}\,g \nonumber \\
		& \quad - 2\,{\rm sym}[\Ric^\sharp(\L_Z\,g)] \,. \label{eq: LieVvec} 
\end{align}	
\end{lemma}	
\begin{remark}
	Note that Lemma \ref{le: crucial equations} and Lemma \ref{baseVac} are immediate consequences of Lemma \ref{base}.
	Note also that in virtue of \eqref{eq: Vvec} equation \eqref{eq: LieVvec} could also be written as 
\begin{align*}
\Box\L_Zg - 2\,\Riem(\L_Zg) + &\L_{\div\,(\L_Z g - \div(Z)\, g)^\sharp}g \\
	& = 2\,\L_Z \Ric - 2\,{\rm sym}[\Ric^\sharp(\L_Z\,g)] \,. 
\end{align*}
This equation will play a central role in generalising our result to the case of various coupled gravity--matter systems following the strategy applied in \cite{Racz1999,Racz2001}.
\end{remark}

\begin{proof}
The proof of \eqref{eq: Vvec} and \eqref{eq: LieVvec} is given by straightforward calculations carried out below by making use of explicit index notation. In doing so our conventions follow those of \cite{wald}.

The first identity comes as
\begin{align*}
	\big[\Box \,Z_{\,\flat} &+ \div(\L_Z g - \div(Z)\, g)\big]{}_b \nonumber\\ 
		& = -\n^a\n_a Z_b + \n^a\left[\left( \n_aZ_b+\n_bZ_a\right) -(\n^eZ_e)\, g_{ab} \right] \nonumber \\ 
		& = \left[\n^a\n_b - \n_b\n^a\right]Z_a \nonumber \\
		& =  R_{\,b}{}^a Z_a \nonumber \\
		& = \big[\Ric(Z)\big]{}_b\,.
\end{align*}
The second identity is somewhat more involved but it is also straightforward. Note that by evaluating $\Box\,\L_Zg$ we get first
\begin{align}
	\big[\Box\,\L_Zg \big]{}_{ab} 
		& = -\n^e\n_e\left( \n_aZ_b+\n_bZ_a\right) \nonumber \\
		& = -g^{ef} \left[\n_e\n_f\left( \n_aZ_b+\n_bZ_a\right)  \right] \nonumber \\ 
		& = -g^{ef} \big[\n_e\left(\n_a\n_fZ_b\right) + \n_e (R_{fab}{}^dZ_d) \nonumber \\*
		& \hphantom{= -g^{ef} \big[} +\n_e\left(\n_b\n_fZ_a\right) + \n_e (R_{fba}{}^dZ_d)   \big] \nonumber \\ 
		&  = -g^{ef} \big[\n_a\left(\n_e\n_fZ_b\right) + R_{eaf}{}^d\n_dZ_b + R_{eab}{}^d\n_fZ_d \nonumber \\*
		& \hphantom{= -g^{ef} \big[} + (\n_eR_{fab}{}^d)Z_d+ R_{fab}{}^d(\n_e Z_d) +\n_b\left(\n_e\n_fZ_a\right) \nonumber \\*
		& \hphantom{= -g^{ef} \big[} + R_{ebf}{}^d\n_dZ_a + R_{eba}{}^d\n_fZ_d +  (\n_eR_{fba}{}^d)Z_d \nonumber \\*
		& \hphantom{= -g^{ef} \big[} + R_{fba}{}^d(\n_e Z_d)   \big] \nonumber \\ 
		&  = -2\,\big[\n_{(a|}\left(\n^e\n_eZ_{|b)}\right) + R_{(a|}{}^d\n_dZ_{|b)} - R_{a}{}^e{}_{b}{}^f\left( \n_eZ_f+\n_fZ_e\right) \nonumber \\*
		& \hphantom{= -2\,\big[} + (\n^eR_{e(ab)}{}^d)Z_d \big]  \nonumber \\ 
		& =\big[\,\L_{\Box\, Z}g + 2\,\riem(\L_Z g)\,\big]{}_{ab} - 2\,R_{(a|}{}^d\n_dZ_{|b)} \nonumber \\*
		&\hphantom{= \big[} - 2\,(\n^eR_{e(ab)}{}^d)Z_d\,. \label{eq: AdabLg0}
\end{align}
The proof is completed once the last two terms are put into some more favorable form. In doing so we shall derive first some useful auxiliary relations. For instance, by a straightforward calculation verifies
\begin{align}\label{eq: LRab}
 \L_Z R_{ab} 
 	& = Z^e\n_e R_{ab} + R_{eb}\n_a Z^e+ R_{ae}\n_b Z^e \nonumber \\  
 	& =  Z^e\n_e R_{ab} + \n_a(R_{eb} Z^e)+ \n_b(R_{ae} Z^e) \nonumber \\
 	& \quad - \left[(\n_a R_{eb}) + (\n_b R_{ae}) \right] Z^e \nonumber \\  
 	& =  \left[\n_e R_{ab} - 2\,\n_{(a|} R_{|b)e} \right] Z^e + \big[\,\L_{\Ric^\sharp(Z)}\,g\,\big]{}_{ab}\,.
\end{align}
By making use then the contracted Bianchi identity 
\begin{equation}\label{eq: bianchi}
\n_e R_{dab}{}^e + \n_d R_{ab} - \n_a R_{bd} =0\,,
\end{equation}
and the symmetries of the Riemann tensor, we get 
\begin{equation*}
(\n^e R_{eab}{}^d)Z_d =  (\n_e R_{a}{}^{ed}{}_{b})Z_d = (\n_e R_{dba}{}^{e})Z^d =  \left[\n_b R_{ad} - \n_d R_{ab}\right] Z^d\,,
\end{equation*}
where in the last step \eqref{eq: bianchi} was used. By combining this last relation with \eqref{eq: LRab} gives then 
\begin{align}
2\,(\n^e R_{e(ab)}{}^d)Z_d 
	&= 2\left[\n_{(a|} R_{|b)d} - \n_d R_{ab}\right] Z^d \nonumber\\
	&= \big[\,\L_{\Ric^\sharp(Z)}\,g - \L_Z \Ric\,\big]{}_{ab} - Z^e\n_e R_{ab}\,. \label{eq: intmed}
\end{align}
Noticing finally that 
\begin{equation}\label{eq: aux}
2\,R_{(a|}{}^d\n_dZ_{|b)} = 2\,R_{(a|}{}^d\L_Z\,g_{d|b)} - R_{a}{}^d\n_bZ_{d} - R_{b}{}^d\n_aZ_{d}
\end{equation}
a combination of \eqref{eq: aux} and \eqref{eq: intmed}, in virtue of the first line of \eqref{eq: LRab}, gives then
\begin{align*}
2\,R_{(a|}{}^d\n_dZ_{|b)} &+ 2\,(\n^eR_{f(ab)}{}^d)Z_d \\
	&= \big[\,\L_{\Ric^\sharp(Z)}\,g - 2\,\L_Z \Ric +  2\,{\rm sym}[\Ric^\sharp(\L_Z\,g)]\,\big]{}_{ab}\,,
\end{align*}
which, along with \eqref{eq: AdabLg0}, completes the verification of \eqref{eq: LieVvec}. 
\end{proof}

\begin{remark}
	We would like to emphasise again that the above computation is free of using any sort of field equation concerning the metric or restrictions on its signature. 
	It would be of interest to find various other applications of the identities \eqref{eq: Vvec} and \eqref{eq: LieVvec}.
\end{remark}

\section*{Acknowledgements}
The authors thank the Albert Einstein Institute in Golm for its hospitality. 
This project emerged from discussions we had there. 
The authors are also deeply indebted to Vince Moncrief for illuminating discussions.
We would also like to thank Ettore Minguzzi and Sebastian Gurriaran for helpful suggestions.
OP would like to thank the Berlin Mathematical School (Deutsche Forschungsgemeinschaft grant no.~GSC 14), Sonderforschungsbereich 647 and Schwerpunktprogramm 2026, funded by Deutsche Forschungsgemeinschaft, for financial support.
IR was supported by NKFIH Grant Nos. K-115434, K-142423 and the POLONEZ program of the National Science Centre of Poland which has received funding from the European Union`s Horizon 2020 research and innovation programme under the Marie Sk{\l}odowska-Curie grant agreement No.~665778.


\end{sloppypar}
\end{document}